\documentclass{cmslatex}
\usepackage[paperwidth=7in, paperheight=10in, margin=.875in]{geometry}
 \usepackage[backref,colorlinks,linkcolor=red,anchorcolor=green,citecolor=blue]{hyperref}
\usepackage{amsfonts,amssymb}
\usepackage{amsmath}
\usepackage{graphicx}
\usepackage{cite}
\usepackage{enumerate}
\renewcommand{\theequation}{\arabic{section}.\arabic{equation}}

   \allowdisplaybreaks
   
   \newcommand{\ds}{\displaystyle}

\begin{document}
 \title{Understanding high-index saddle dynamics via numerical analysis\thanks{Received date, and accepted date (The correct dates will be entered by the editor).}}


          \author{Lei Zhang\thanks{Beijing International Center for Mathematical Research, Center
for Machine Learning Research, Center for Quantitative Biology, Peking University, Beijing, 100871, China (zhangl@math.pku.edu.cn)}
          \and Pingwen Zhang\thanks{School of Mathematics and Statistics, Wuhan University,
Wuhan, 430072, China; School of Mathematical Sciences, Laboratory of Mathematics
and Applied Mathematics, Peking University, Beijing, 100871, China
 (pzhang@pku.edu.cn)}
 \and Xiangcheng Zheng\thanks{School of Mathematics, Shandong University, Jinan, 250100,
China (xzheng@sdu.edu.cn)}}

         \pagestyle{myheadings} \markboth{Understanding HiSD via numerical analysis}{Zhang, Zhang and Zheng} \maketitle

          \begin{abstract}
             High-index saddle dynamics (HiSD) serves as a competitive instrument in searching the any-index saddle points and constructing the solution landscape of complex systems. The Lagrangian multiplier terms in HiSD ensure the Stiefel manifold constraint, which, however, are dropped in the commonly-used discrete HiSD scheme and are replaced by an additional Gram-Schmidt orthonormalization. Though this scheme has been successfully applied in various fields, it is still unclear why the above modification does not affect its effectiveness. We recover the same form as HiSD from this scheme, which not only leads to error estimates naturally, but indicates that the mechanism of Stiefel manifold preservation by Lagrangian multiplier terms in HiSD is nearly a Gram-Schmidt process (such that the above modification is appropriate). The developed methods are further extended to analyze the more complicated constrained HiSD on high-dimensional sphere, which reveals more mechanisms of the constrained HiSD in preserving several manifold properties.
          \end{abstract}

      
      \section{Introduction}
Searching saddle points on a complicated energy landscape is a hot but challenging topic in computational physical and chemistry \cite{EV2010,Han2019transition,Mehta,wang2021modeling,
Yin2020nucleation,npj2016}. The saddle points can be classified by the (Morse) index, which, according to the Morse theory \cite{Milnor}, are characterized by the maximal dimension of a subspace on which the Hessian is negative definite. 
There exist extensive searching algorithms for saddle points \cite{Doye,EZho,Farr,Gao,Gou,Gra,Lev,Li2001,LiZho,LiuXie,Xie,
ZhaDu,ZhaSISC}. 
This work focuses on a high-index saddle dynamics (HiSD) approach \cite{YinSISC} for finding an index-$k$ saddle point of the energy functional $E(x)$ and constructing solution landscapes \cite{Han2021,YinPRL,YinSCM,YuZhaZha}
\begin{equation}\label{sadk}
\left\{
\begin{array}{l}
\ds \frac{dx}{dt} =\mathcal S(t),\\[0.1in]
\ds \frac{dv_i}{dt}=\mathcal R_i(t)+\mathcal L_i(t),    ~~1\leq i\leq k,
\end{array}
\right.
\end{equation}
where
\begin{equation}\label{sadkaux}
\left\{
\begin{array}{l}
\ds \mathcal S(t):=\beta\bigg(I -2\sum_{j=1}^k v_jv_j^\top \bigg)F(x),\\[0.175in]
\ds  \mathcal R_i(t):=  \gamma J(x)v_i,\\[0.075in]
\ds \mathcal L_i(t):= \gamma\bigg( -v_iv_i^\top-2\sum_{j=1}^{i-1}v_jv_j^\top\bigg)J(x)v_i.
\end{array}
\right.
\end{equation}
Here $x\in\mathbb R^d$ represents the state variable, $v_i (i=1,\cdots,k)$ are $k$ directional variables constructing the unstable subspace of the target saddle point, $F(x)=-\nabla E(x)$, $J(x)=-\nabla^2 E(x)$, and $\beta$, $\gamma>0$ are relaxation parameters. 
It is shown in \cite{YinSISC} that a linearly stable steady state of (\ref{sadk}) is an index-$k$ saddle
point. From the original derivations of HiSD in \cite{YinSISC}, the $\mathcal R_i(t)$ arises from minimizing the Rayleigh quotient, while $\mathcal L_i(t)$ is introduced via the Lagrangian multiplier method to ensure the Stiefel manifold constraint, that is, the orthonormality of directional vectors $\{v_i(t)\}_{i=1}^k$ for any $t>0$, provided that the initial values $\{v_i(0)\}_{i=1}^k$ are orthonormal.
\subsection{Motivation}
 An efficient algorithm for HiSD is developed in  \cite{YinSCM} with numerical solutions $\{x_n\}_{n=1}^N$ and $\{v_{i,n}\}_{i=1,n=1}^{k,N}$ 
 \begin{equation}\label{FDsadk}
\left\{
\begin{array}{l}
\ds x_n=x_{n-1}+\tau \mathcal S^{n-1},\\[0.1in]
\ds \tilde v_{i,n}=v_{i,n-1}+\tau\mathcal R_i^{n-1},~~1\leq i\leq k,\\[0.1in]
\ds  v_{i,n}=\text{GramSchmidt}(v_{1,n},
\cdots,v_{i-1,n};\tilde v_{i,n}),~~1\leq i\leq k,
\end{array}
\right.
\end{equation}
equipped with the initial state $x_0$ and orthonormal initial directional vectors $\{v_{i,0}\}_{i=1}^k$, where
 \begin{equation}\label{FDsadkaux}
 \begin{array}{l}
\ds \mathcal S^{n-1}:=\beta\bigg(I -2\sum_{j=1}^k v_{j,n-1}v_{j,n-1}^\top \bigg)F(x_{n-1}),\\[0.175in]
\ds \mathcal R_i^{n-1}:=\gamma J(x_{n-1})v_{i,n-1}.
\end{array}
\end{equation}
We observe that the Lagrangian multiplier terms in $\mathcal L_i$ in HiSD are dropped in this scheme and the Gram-Schmidt orthonormalization is thus critical to enforce the Stiefel manifold constraint.

A related work \cite{Z3} analyzes the algorithm (\ref{FDsadk}) with the second scheme replaced by 
\begin{equation}\label{FDmh}
\begin{array}{l}
\ds\tilde v_{i,n}=v_{i,n-1}+\tau\mathcal R_i^{n-1}+\tau\mathcal L_i^{n-1}\text{ where}\\[0.05in]
\ds\mathcal L_i^{n-1}:=\gamma\bigg( -v_{i,n-1}v_{i,n-1}^\top-2\sum_{j=1}^{i-1}v_{j,n-1}v_{j,n-1}^\top\bigg)J(x_{n-1})v_{i,n-1},
\end{array}
\end{equation}
which leads to the scheme proposed in the original work \cite{YinSISC}. In comparison with (\ref{FDsadk}), the Lagrangian multiplier terms in $\mathcal L_i$ in HiSD are reserved such that (\ref{FDmh}) is the exact discretization of the equation of $v_i$ in (\ref{sadk}) and the Gram-Schmidt orthonormalization serves as a perturbation that retracts the dynamics of directional vectors to the Stiefel manifold. For this reason, a perturbation analysis is carried out in \cite{Z3} to perform error estimates, which ensures that the the numerical scheme evolves along the dynamical pathway of continuous HiSD such that the numerical scheme also converges to the same target saddle point of HiSD. Other numerical treatments such as the projection methods for differential equations on manifolds \cite{Haibook}, which project the dynamics of directional vectors back to the Stiefel manifold at each time step, could also be applied with error estimates derived from the conclusions in \cite{Haibook}.

However, numerical analysis for the scheme (\ref{FDsadk}) could not follow the aforementioned ones since the discrete dynamics of directional vectors is not consistent with its continuous analogues. Due to the loss of Lagrangian multiplier terms in (\ref{FDsadk}), the Gram-Schmidt orthonormalization in (\ref{FDsadk}) is no longer a perturbation or projection but may impose a substantial adjustment to enforce the Stiefel manifold constraint as the Lagrangian multiplier terms do in the continuous HiSD. In order to understand the effectiveness of the scheme shown in \cite{YinSCM} and ensure its convergence to the same target saddle point as continuous HiSD, it is natural to investigate whether these modifications in numerical discretization change the mechanisms of HiSD in preserving manifold properties and to what extent deviate the numerical solutions from the latent trajectory of HiSD. 

\subsection{Contribution}\label{sec1.2}
The main contributions of this work are enumerated to address the aforementioned issues:
\begin{itemize}
 \item[\textbf{(i)}] We prove that the dynamics of directional vectors in (\ref{sadk}) could be recovered from the superposition of the discrete dynamics of minimizing the Rayleigh quotient and the Gram-Schmidt orthonormalization, i.e. the second and the third equations in (\ref{FDsadk}), respectively, with the error of order $O(\tau)$ (cf. Theorem \ref{thmvv}). Several novel splittings such as (\ref{split}) and the subsequent estimates of (\ref{mhbest}) are proposed to explore the hidden structures of the Gram-Schmidt process and to gradually get over the nonlinearity and coupling.   This result not only reduces the error estimate of (\ref{FDsadk}) to that for standard system of differential equations, but reveals that the mechanism of Stiefel manifold preservation in HiSD is close to the Gram-Schmidt process, which improves the understanding of HiSD via numerical analysis.

\item[\textbf{(ii)}] We extend the results for the constrained HiSD on the unit sphere $S^{d-1}$ \cite{ZhaDuJCP,CHiSD2021}, which has been successfully applied in computing constrained saddle points of, e.g. the Bose-Einstein condensation \cite{Bao,BaoDu} 
\begin{equation}\label{csadk}
\left\{
\begin{array}{l}
\ds \frac{dx}{dt} =\mathcal S(t)  -xx^\top F(x),\\[0.1in]
\ds \frac{dv_i}{dt}=\mathcal R_i(t)+\mathcal L_i(t)  -xx^\top J(x)v_i+ xv_i^\top F(x),~~1\leq i\leq k,
\end{array}
\right.
\end{equation}
where $\mathcal S$, $\mathcal R_i$ and $\mathcal L_i$ are defined as before with relaxation parameters $\beta=\gamma=1$ for simplicity.
Specifically, \textbf{(a)} we prove that the dynamics of the state variable in (\ref{csadk}) could be recovered from the superposition of the discrete unconstrained discrete gradient dynamics (i.e. the first scheme of (\ref{FDsadk})) and the retraction via the vector normalization, while \textbf{(b)} the dynamics of directional vectors could be recovered from the superposition of the discrete dynamics of minimizing the Rayleigh quotient, the vector transport and the Gram-Schmidt orthonormalization, with the error of order $O(\tau)$ (cf. Theorem \ref{thmc}).  Similar to \textbf{(i)},  these results could significantly simplify the error estimate of (\ref{FDsadk}) and, more importantly, reveal that the mechanisms of the constrained HiSD on preserving several manifold properties (\ref{prop}) are close to the simple operations such as the vector normalization, the vector transport and the Gram-Schmidt orthonormalization.
\end{itemize}

\section{Recovery of HiSD}
The main purpose of this section is to prove that the dynamics of directional vectors could be recovered by combining the second and the third equations in the scheme (\ref{FDsadk}), except for high-order perturbations. This result not only demonstrates the statements in \textbf{(i)}, but will facilitate error estimates in subsequent sections. 

We make the assumptions following \cite{Z3,Z3c}:

\noindent\textbf{Assumption A:} There exists a constant $L>0$ such that the following linearly growth and Lipschitz conditions hold under the standard $l^2$ norm $\|\cdot\|$ of a vector or a matrix
$$\begin{array}{c}
\ds \|J(x_2)-J(x_1)\|+\|F(x_2)-F(x_1)\|\leq L\|x_2-x_1\|,\\[0.1in]
\ds\|F(x)\|\leq L(1+\|x\|),~~x,x_1,x_2\in \mathbb R^d.
\end{array}  $$
It is shown in \cite{Z3} that, under the Assumption A, $\|x_n\|$ is bounded by some fixed constant for $1\leq n\leq N$, which, based on the scheme of $\tilde v_{i,n}$ in (\ref{FDsadk}), implies the boundedness of $\|\tilde v_{i,n}\|$. Furthermore, according to the formula of the Gram-Schmidt procedure, the third equation of (\ref{FDsadk}) could be written in a clearer manner
\begin{equation}\label{vgs}
v_{i,n}=\frac{1}{Y_{i,n}}\bigg(\ds\tilde v_{i,n}-\sum_{j=1}^{i-1}(\tilde v_{i,n}^\top v_{j,n})v_{j,n}\bigg)
\end{equation}
where 
\begin{equation}\label{Yv}
\ds Y_{i,n}:=\bigg\|\tilde v_{i,n}-\sum_{j=1}^{i-1}(\tilde v_{i,n}^\top v_{j,n})v_{j,n}\bigg\|=\bigg(\|\tilde v_{i,n}\|^2-\sum_{j=1}^{i-1}(\tilde v_{i,n}^\top v_{j,n})^2\bigg)^{1/2}.
\end{equation} 
This explicit formula will be frequently used as the third equation of (\ref{FDsadk}) in the following derivations.

We first prove a preliminary estimate for the difference $v_{i,n}-v_{i,n-1}$ for future use. Throughout the paper we use $Q$ to denote a generic positive constant that may assume difficult values at different occurrences.

\begin{lemma}\label{lemvv}
For $\tau$ small enough the following estimate holds for $1\leq i\leq k$ and $1\leq n\leq N$
\begin{equation}\label{vnvn1}
 \|v_{i,n}-v_{i,n-1}\|\leq  Q\tau.
 \end{equation}
Here $Q$ is independent from $i$, $\tau$ and $N$.
\end{lemma}
\begin{proof}
We first prove the conclusion for $i=1$. From the second and the third equations of (\ref{FDsadk}) with $i=1$ we obtain
$$\begin{array}{l}
\ds v_{1,n}=\frac{\tilde v_{1,n}}{\|\tilde v_{1,n}\|}=\tilde v_{1,n}+\frac{\tilde v_{1,n}}{\|\tilde v_{1,n}\|}(1-\|\tilde v_{1,n}\|)\\
\ds\qquad=v_{1,n-1}+\tau\gamma J(x_{n-1})v_{1,n-1}+\frac{\tilde v_{1,n}}{\|\tilde v_{1,n}\|}(1-\|\tilde v_{1,n}\|),
\end{array}  $$
which implies
$$
\ds \|v_{1,n}-v_{1,n-1}\|\leq Q\tau+|1-\|\tilde v_{1,n}\||.
 $$
We incorporate this with
$$\|\tilde v_{1,n}\|=\|v_{1,n-1}+\tau\gamma J(x_{n-1})v_{1,n-1}\|=1+O(\tau) $$
to get $
 \|v_{i,n}-v_{i,n-1}\|\leq  Q_1\tau
 $
for some positive constant $Q_1$. Then we assume that 
\begin{equation}\label{hyp}
 \|v_{j,n}-v_{j,n-1}\|\leq  Q_j\tau
 \end{equation}
for $1\leq j\leq i-1$ for some $1\leq i\leq k$ and for some positive constants $Q_1,\cdots,Q_{i-1}$. We intend to prove that
$$
 \|v_{i,n}-v_{i,n-1}\|\leq  Q_i\tau
$$
for some positive constant $Q_i$. Here $Q_i$ could be greater than $Q_1,\cdots,Q_{i-1}$.
We invoke the second equation of (\ref{FDsadk}) in the third equation of (\ref{FDsadk}) to obtain
\begin{equation}\label{vexp}
\begin{array}{l}
\ds v_{i,n}=\frac{1}{Y_{i,n}}\bigg(v_{i,n-1}+\tau\gamma J(x_{n-1})v_{i,n-1}-\sum_{j=1}^{i-1}(v_{i,n-1}^\top v_{j,n})v_{j,n}\\
\ds\qquad\qquad-\tau\gamma\sum_{j=1}^{i-1}(v_{i,n-1}^\top J(x_{n-1})^\top v_{j,n})v_{j,n}\bigg).
\end{array}
\end{equation}
We apply $v_{i,n-1}^\top v_{j,n-1}=0$ to obtain
\begin{equation}
\begin{array}{l}
\ds v_{i,n}-v_{i,n-1}=\frac{1}{Y_{i,n}}\bigg(v_{i,n-1}(1-Y_{i,n})+\tau\gamma J(x_{n-1})v_{i,n-1}\\
\ds\qquad-\sum_{j=1}^{i-1}\big(v_{i,n-1}^\top (v_{j,n}-v_{j,n-1})\big)v_{j,n}-\tau\gamma\sum_{j=1}^{i-1}(v_{i,n-1}^\top J(x_{n-1})^\top v_{j,n})v_{j,n}\bigg),
\end{array}
\end{equation}
which leads to
\begin{equation}\label{v-v}
\begin{array}{l}
\ds \|v_{i,n}-v_{i,n-1}\|\leq\frac{1}{Y_{i,n}}\bigg(|1-Y_{i,n}|+\sum_{j=1}^{i-1}\|v_{j,n}-v_{j,n-1}\|+Q\tau\bigg).
\end{array}
\end{equation}
As for $\tau$ small enough
\begin{equation}\label{Y} \begin{array}{l}
\ds Y_{i,n}=\bigg(\|\tilde v_{i,n}\|^2-\sum_{j=1}^{i-1}(\tilde v_{i,n}^\top v_{j,n})^2\bigg)^{1/2}\\
\ds\qquad=\bigg(\|v_{i,n-1}+\tau\gamma J(x_{n-1})v_{i,n-1}\|^2\\
\ds\qquad\quad-\sum_{j=1}^{i-1}\big((v_{i,n-1}+\tau\gamma J(x_{n-1})v_{i,n-1})^\top v_{j,n}\big)^2\bigg)^{1/2}\\
\ds\qquad=\bigg(1+2\tau\gamma v_{i,n-1}^\top J(x_{n-1})v_{i,n-1}+O(\tau^2)\\
\ds\qquad\quad-\sum_{j=1}^{i-1}\big(v_{i,n-1}^\top (v_{j,n}-v_{j,n-1})+\tau\gamma v_{i,n-1}^\top J(x_{n-1})^\top v_{j,n}\big)^2\bigg)^{1/2}\\
\ds\qquad\in \bigg[1\pm Q\bigg(\sum_{j=1}^{i-1} \|v_{j,n}-v_{j,n-1}\|^2+\tau\bigg)\bigg]^{1/2},
\end{array} \end{equation}
we obtain
\begin{equation}\label{1-Y}
|1-Y_{i,n}|\leq |1-Y_{i,n}^2|\leq Q\bigg(\sum_{j=1}^{i-1} \|v_{j,n}-v_{j,n-1}\|^2+\tau\bigg).
\end{equation}
We incorporate this estimate with (\ref{v-v}) to obtain
\begin{equation}\label{v-v2}
 \ds \|v_{i,n}-v_{i,n-1}\|\leq\frac{\ds Q\sum_{j=1}^{i-1} \|v_{j,n}-v_{j,n-1}\|^2+\sum_{j=1}^{i-1}\|v_{j,n}-v_{j,n-1}\|+Q\tau}{\ds\bigg(1-Q\bigg(\sum_{j=1}^{i-1} \|v_{j,n}-v_{j,n-1}\|^2+\tau\bigg)\bigg)^{1/2}}, 
 \end{equation}
which, together with the hypothesis (\ref{hyp}), leads to  
 \begin{equation*}
 \ds \|v_{i,n}-v_{i,n-1}\|\leq\frac{ Q\tau^2+Q\tau}{\ds\big(1-Q(\tau^2+\tau)\big)^{1/2}}\leq Q_i\tau. 
 \end{equation*}
 Thus we obtain (\ref{hyp}) for $j=i$, which completes the induction procedure. Then we select $Q$ in (\ref{vnvn1}) as $\max\{Q_1,\cdots,Q_k\}$ to complete the proof.
\end{proof}

We then prove the main theorem of this section.
\begin{theorem}\label{thmvv}
For $\tau$ small enough, combining the second and the third equations in (\ref{FDsadk}), which correspond to the discrete dynamics of minimizing the Rayleigh quotient and the Gram-Schmidt procedure, respectively, leads to the discrete dynamics of directional vectors in (\ref{sadk}) for $1\leq n\leq N$ and $1\leq i\leq k$
\begin{equation}\label{ee1} 
\begin{array}{l}
 \ds \frac{v_{i,n}-v_{i,n-1}}{\tau}=\mathcal R_i^{n-1}+\mathcal L_i^{n-1}+O(\tau).
\end{array}  
\end{equation}
\end{theorem}
\begin{remark}\label{rem31}
From this theorem we observe that the Lagrangian multiplier terms are recovered in the second equation of (\ref{FDsadk}) by invoking the third equation of (\ref{FDsadk}) such that, expect for the error $O(\tau)$,  (\ref{ee1}) is exactly the explicit numerical scheme of the equation of $v_i$ in (\ref{sadk}). As $\tau$ tends to 0, (\ref{ee1}) and thus the superposition of  the second and the third equations in (\ref{FDsadk}) converges to the dynamics of directional vectors in HiSD, which may indicate that the Gram-Schmidt process  has the same effects as the Lagrangian multiplier terms that justifies the claims in \textbf{(i)}. 

Furthermore, in error estimates we could easily generate the error equations by subtracting the reference equation of $v_i$ from (\ref{ee1}). In other words, (\ref{ee1}) provides a much more feasible form to generate the error equations than the original scheme (i.e. the second and the third equations in (\ref{FDsadk})).
\end{remark}
\begin{proof}
From the last-but-one equality of (\ref{Y}) and Lemma \ref{lemvv}, we have
$$Y_{i,n}=\big(1+2\tau\gamma v_{i,n-1}^\top J(x_{n-1})v_{i,n-1}+O(\tau^2)\big)^{1/2}. $$
Then we introduce a novel splitting
\begin{equation}\label{split}
\begin{array}{l}
\ds \frac{1}{Y_{i,n}}=1+\frac{1-Y_{i,n}^2}{Y_{i,n}(1+Y_{i,n})}\\[0.15in]
\qquad\,\ds=1+\frac{-2\tau\gamma v_{i,n-1}^\top J(x_{n-1})v_{i,n-1}+O(\tau^2)}{Y_{i,n}(1+Y_{i,n})}\\[0.15in]
\ds\qquad\,=1-\tau\gamma v_{i,n-1}^\top J(x_{n-1})v_{i,n-1}\\[0.05in]
\ds\qquad\qquad-2\tau\gamma v_{i,n-1}^\top J(x_{n-1})v_{i,n-1}\bigg(\frac{1}{Y_{i,n}(1+Y_{i,n})}-\frac{1}{2}\bigg)\\
\ds\qquad\qquad+\frac{O(\tau^2)}{Y_{i,n}(1+Y_{i,n})}.
\end{array}
\end{equation}
By (\ref{1-Y}) and Lemma \ref{lemvv}, the $(\cdots)$ term in the last-but-one right-hand side term of (\ref{split}) could be estimated as
\begin{equation}\label{YY} 
 \begin{array}{l}
\ds \bigg|\frac{1}{Y_{i,n}(1+Y_{i,n})}-\frac{1}{2}\bigg|=\frac{|1-Y_{i,n}|}{1+Y_{i,n}}\bigg(\frac{1}{Y_{i,n}}+\frac{1}{2}\bigg)\leq Q|1-Y_{i,n}|\leq Q\tau.
\end{array} 
 \end{equation}
Thus the last-but-one right-hand side term of (\ref{split}) is indeed an $O(\tau^2)$ term, and we invoke this in (\ref{split}) to obtain
\begin{equation}\label{split2}
\begin{array}{l}
\ds \frac{1}{Y_{i,n}}=1-\tau\gamma v_{i,n-1}^\top J(x_{n-1})v_{i,n-1}+O(\tau^2).
\end{array}
\end{equation}
We substitute $1/Y_{i,n}$ in (\ref{vexp}) by this equation to obtain
\begin{equation}\label{mhbest}
\begin{array}{l}
\ds v_{i,n}=\big(1-\tau\gamma v_{i,n-1}^\top J(x_{n-1})v_{i,n-1}+O(\tau^2)\big)\\[0.05in]
\ds \qquad\qquad\times\bigg(v_{i,n-1}+\tau\gamma J(x_{n-1})v_{i,n-1}-\sum_{j=1}^{i-1}(v_{i,n-1}^\top v_{j,n})v_{j,n}\\
\ds\qquad\qquad\qquad\qquad-\tau\gamma\sum_{j=1}^{i-1}(v_{i,n-1}^\top J(x_{n-1})^\top v_{j,n})v_{j,n}\bigg)\\
\ds\qquad=v_{i,n-1}+\tau\gamma J(x_{n-1})v_{i,n-1}-\sum_{j=1}^{i-1}(v_{i,n-1}^\top v_{j,n})v_{j,n}\\
\ds\qquad\qquad-\tau\gamma\sum_{j=1}^{i-1}(v_{i,n-1}^\top J(x_{n-1})^\top v_{j,n})v_{j,n}\\[0.2in]
\ds\qquad\qquad-\tau\gamma v_{i,n-1}^\top J(x_{n-1})v_{i,n-1}v_{i,n-1}\\[0.05in]
\ds\qquad\qquad+\tau\gamma\sum_{j=1}^{i-1}(v_{i,n-1}^\top v_{j,n})v_{i,n-1}^\top J(x_{n-1})v_{i,n-1}v_{j,n}+O(\tau^2)\\[0.1in]
\ds\qquad=:\sum_{m=1}^6 A_m+O(\tau^2).
\end{array}
\end{equation}
From the definition of $\tilde v_{j,n}$ for $1\leq j\leq i-1$, we have
\begin{equation}\label{hui}
 v_{i,n-1}^\top \tilde v_{j,n}=v_{i,n-1}^\top(v_{j,n-1}+\tau\gamma J(x_{n-1})v_{j,n-1})=\tau\gamma v_{i,n-1}^\top J(x_{n-1})v_{j,n-1}.
 \end{equation}
We apply this to rewrite $A_3$ as
$$\begin{array}{l}
\ds A_3=-\sum_{j=1}^{i-1}(v_{i,n-1}^\top v_{j,n})v_{j,n}\\[0.175in]
\ds\quad~=-\sum_{j=1}^{i-1}v_{i,n-1}^\top (v_{j,n}-\tilde v_{j,n})v_{j,n}-\sum_{j=1}^{i-1}v_{i,n-1}^\top \tilde v_{j,n}v_{j,n}\\[0.175in]
\ds\quad~=-\sum_{j=1}^{i-1}v_{i,n-1}^\top (v_{j,n}-\tilde v_{j,n})v_{j,n}-\tau\gamma\sum_{j=1}^{i-1}v_{i,n-1}^\top J(x_{n-1})v_{j,n-1} v_{j,n}\\[0.175in]
\ds\quad~=:A_{3,1}+A_{3,2}.
\end{array} $$
By Lemma \ref{lemvv}, $A_{3,2}$ could be reformulated as
\begin{equation}\label{A32} \begin{array}{l}
\ds A_{3,2}=-\tau\gamma\sum_{j=1}^{i-1}v_{i,n-1}^\top J(x_{n-1})v_{j,n-1} v_{j,n-1}\\[0.15in]
\ds\qquad\qquad+\tau\gamma\sum_{j=1}^{i-1}v_{i,n-1}^\top J(x_{n-1})v_{j,n-1} (v_{j,n-1}-v_{j,n})\\[0.15in]
\ds\qquad\,=-\tau\gamma\sum_{j=1}^{i-1}v_{i,n-1}^\top J(x_{n-1})v_{j,n-1} v_{j,n-1}+O(\tau^2).
\end{array} \end{equation}
To estimate $A_{3,1}$, from the third equation of (\ref{FDsadk}), we have
\begin{equation}
\ds v_{j,n}-\tilde v_{j,n}=\bigg(\frac{1}{Y_{j,n}}-1\bigg)\tilde v_{j,n}-\frac{1}{Y_{j,n}}\sum_{l=1}^{j-1}(\tilde v_{j,n}^\top v_{l,n})v_{l,n},
\end{equation}
which implies
\begin{equation}\label{hui2}
\ds v_{i,n-1}^\top (v_{j,n}-\tilde v_{j,n})=\bigg(\frac{1}{Y_{j,n}}-1\bigg)v_{i,n-1}^\top\tilde v_{j,n}-\frac{1}{Y_{j,n}}\sum_{l=1}^{j-1}\tilde v_{j,n}^\top v_{l,n} v_{i,n-1}^\top v_{l,n}.
\end{equation}
From (\ref{split2}) and (\ref{hui}), the first right-hand side term of (\ref{hui2}) is an $O(\tau^2)$ term, while the second right-hand side term could be reformulated as
\begin{equation}\label{hui3} \begin{array}{l}
\ds  -\frac{1}{Y_{j,n}}\sum_{l=1}^{j-1}\tilde v_{j,n}^\top v_{l,n} v_{i,n-1}^\top v_{l,n}\\[0.15in]
\ds\quad=-\frac{1}{Y_{j,n}}\sum_{l=1}^{j-1}\tilde v_{j,n}^\top (v_{l,n}-v_{l,n-1}) v_{i,n-1}^\top v_{l,n}\\[0.2in]
\ds\qquad-\frac{1}{Y_{j,n}}\sum_{l=1}^{j-1}\tilde v_{j,n}^\top v_{l,n-1} v_{i,n-1}^\top v_{l,n}.
\end{array} \end{equation}
By Lemma \ref{lemvv}, $v_{i,n-1}^\top v_{l,n}=v_{i,n-1}^\top (v_{l,n}-v_{l,n-1})$ is an $O(\tau)$ term, and 
$$ \tilde v_{j,n}^\top v_{l,n-1}=v_{l,n-1}^\top (v_{j,n-1}+\tau\gamma J(x_{n-1})v_{j,n-1})=\tau\gamma v_{l,n-1}^\top J(x_{n-1})v_{j,n-1}$$
is also an $O(\tau)$ term. Thus, (\ref{hui3}) is an $O(\tau^2)$ term, which implies (\ref{hui2}) is also an $O(\tau^2)$ term. Consequently, $A_{3,1}$ is an $O(\tau^2)$ term, which, together with (\ref{A32}), leads to
\begin{equation}\label{a3}
 A_3=-\tau\gamma\sum_{j=1}^{i-1}v_{i,n-1}^\top J(x_{n-1})v_{j,n-1} v_{j,n-1}+O(\tau^2). 
 \end{equation}
 We then split $A_4$ as
 \begin{equation*}
 \begin{array}{l}
 \ds A_4=-\tau\gamma\sum_{j=1}^{i-1}(v_{i,n-1}^\top J(x_{n-1})^\top v_{j,n-1})v_{j,n-1}\\[0.15in]
 \ds\qquad-\tau\gamma\sum_{j=1}^{i-1}(v_{i,n-1}^\top J(x_{n-1})^\top (v_{j,n}-v_{j,n-1}))v_{j,n-1}\\[0.15in]
 \ds\qquad-\tau\gamma\sum_{j=1}^{i-1}(v_{i,n-1}^\top J(x_{n-1})^\top v_{j,n})(v_{j,n}-v_{j,n-1}) .
 \end{array}
 \end{equation*}
 By Lemma \ref{lemvv} we obtain
\begin{equation}\label{a4}
A_4=-\tau\gamma\sum_{j=1}^{i-1}(v_{i,n-1}^\top J(x_{n-1})^\top v_{j,n-1})v_{j,n-1}+O(\tau^2).
\end{equation} 
By the symmetry of $J$, we incorporate (\ref{a3}) and (\ref{a4}) to find that
 $$\begin{array}{l}
 \ds A_1+\cdots+ A_5=v_{i,n-1}+\tau\gamma J(x_{n-1})v_{i,n-1}\\[0.05in]
\ds\qquad\qquad\qquad\qquad-\tau\gamma v_{i,n-1}^\top J(x_{n-1})v_{i,n-1}v_{i,n-1}\\[0.05in]
\ds\qquad\qquad\qquad\qquad-2\tau\gamma\sum_{j=1}^{i-1}(v_{i,n-1}^\top J(x_{n-1}) v_{j,n-1})v_{j,n-1}+O(\tau^2).
\end{array}  $$
 Therefore, in order to get (\ref{ee1}), we need to show that $A_6=O(\tau^2)$. As
 $$\begin{array}{l}
 \ds A_6=\tau\gamma\sum_{j=1}^{i-1}(v_{i,n-1}^\top v_{j,n})v_{i,n-1}^\top J(x_{n-1})v_{i,n-1}v_{j,n}\\
 \ds\quad~=\tau\gamma\sum_{j=1}^{i-1}v_{i,n-1}^\top (v_{j,n}-v_{j,n-1})v_{i,n-1}^\top J(x_{n-1})v_{i,n-1}v_{j,n},
\end{array}   $$
 we apply Lemma \ref{lemvv} again to find that $A_6$ is an $O(\tau^2)$ term, which completes the proof. 
\end{proof}

\section{Error estimates and numerical experiments}\label{sec4}
Based on Theorem \ref{thmvv}, we prove error estimates for the numerical scheme (\ref{FDsadk}) and perform numerical experiments to substantiate the theoretical findings.
\subsection{Error estimates}\label{sec41}
The error equation of $e^x_n$ could be generated by subtracting the first equation of (\ref{FDsadk}) from the reference equation of $x(t)$, which is obtained by discretizing the first equation of (\ref{sadk}) via the Euler discretization 
$$x(t_{n}) =x(t_{n-1})+\tau\mathcal S(t_{n-1})+O(\tau^2). $$
The error equation of $e^{v_i}_n$ could be derived by subtracting (\ref{ee1}) from the reference equation of $v_i(t)$
$$\begin{array}{l}
\ds v_{i}(t_n)=v_{i}(t_{n-1})+\tau\mathcal R_i(t_{n-1})+\tau\mathcal L_i(t_{n-1})+O(\tau^2).
\end{array}  $$
Based on these error equations, the error estimates could be performed following those for standard system of differential equations \cite{BaoCao}, and we thus directly state the result in the following theorem.
\begin{theorem}\label{thmevk}
Under the Assumption A, the following estimate holds for the scheme (\ref{FDsadk}) for $\tau$ sufficiently small 
\begin{equation}\label{z3est}
\|x(t_n)-x_{n}\|+\sum_{i=1}^k \|v_i(t_n)-v_{i,n}\|\leq Q\tau,~~1\leq n\leq N. 
\end{equation}
Here $Q$ is independent from $\tau$, $n$ and $N$.
\end{theorem}
\begin{remark}\label{rem1}
Let $\{X_n\}_{n=1}^N$ and $\{V_{i,n}\}_{i=1,n=1}^{k,N}$ be numerical solutions of the scheme (\ref{FDsadk}) with the second equation replaced by (\ref{FDmh}), i.e. the numerical discretization scheme in \cite{Z3}. According to \cite{Z3} the following estimates hold
$$
\|x(t_n)-x_n\|+\sum_{i=1}^k \|v_i(t_n)-v_{i,n}\|\leq Q\tau,~~1\leq n\leq N,
$$
which, together with (\ref{z3est}), leads the following estimate between numerical solutions of different schemes 
$$\begin{array}{l}
\ds \|x_n-X_n\|+\sum_{i=1}^k\|v_{i,n}-V_{i,n}\|\\[0.175in]
\ds\quad\leq \|x(t_n)-x_n\|+\|x(t_n)-X_n\|\\[0.05in]
\ds\qquad+\sum_{i=1}^k\big(\|v_i(t_n)-v_{i,n}\|+\|v_i(t_n)-V_{i,n}\|\big)\leq Q\tau.
\end{array}  $$
This implies that the difference between the numerical solutions turns to zero as $\tau$ decreases such that both methods generate almost the same numerical solutions for $\tau$ small enough. Nevertheless, the dynamics of directional vectors in (\ref{FDsadk}) saves $O(d^2N k^2)$ (or $O(dN k^2)$ if the dimer method \cite{Dimer} could be used to approximate the product of the Hessian matrix and the vector) computational cost in comparison with the scheme (\ref{FDmh}) that significantly improves the computational efficiency for large $N$, $d$ or $k$.
\end{remark}
\subsection{Numerical experiments}
We carry out numerical experiments to test the convergence rate (denoted by ``CR'' in tables) of the numerical scheme (\ref{FDsadk}) and compare the behavior between (\ref{FDsadk}) and the scheme in \cite{Z3}. We consider the following two-dimensional system proposed in \cite{EZho}
\begin{equation}\label{ezho}
 E(x,y)=-\frac{1}{4}(x^2-1)^2-\frac{1}{2}y^2. 
 \end{equation}
For this system $(0,0)$ is an index-1 saddle point and $(1,0)$ is an index-2 saddle point.

\textbf{Example 1: Accuracy test}
We compute the index-1 saddle point of (\ref{ezho}) with the initial conditions
$$x_0=
\left[\!\!
\begin{array}{c}
\ds 1\\
\ds 0.5
\end{array}
\!\!\right],~~
v_{1,0}=\frac{1}{\sqrt{2}}
\left[\!\!
\begin{array}{c}
\ds -1\\
\ds -1
\end{array}
\!\!\right]
 $$
 and the index-2 saddle point with the initial conditions 
$$x_0=
\left[\!\!
\begin{array}{c}
\ds 1.3\\
\ds 0.5
\end{array}
\!\!\right],~~
v_{1,0}=\frac{1}{\sqrt{2}}
\left[\!\!
\begin{array}{c}
\ds -1\\
\ds -1
\end{array}
\!\!\right],~~
v_{2,0}=\frac{1}{\sqrt{5}}
\left[\!\!
\begin{array}{c}
\ds -2\\
\ds -1
\end{array}
\!\!\right].
 $$ 
As the exact solutions to the high-index saddle dynamics are not available, numerical solutions computed under $\tau=2^{-13}$ serve as the reference solutions. We set $\beta=\gamma=1$ and $T=7$ to ensure that the saddle dynamics reaches the target saddle point. Numerical results are presented in Tables \ref{table1:1}-\ref{table1:2}, which demonstrate the first-order accuracy of the numerical scheme (\ref{FDsadk}) as proved in Theorem \ref{thmevk}.
\begin{table}[h]
\setlength{\abovecaptionskip}{0pt}
\centering
\caption{Convergence rates of computing index-1 saddle point.}
{\small
\begin{tabular}{c|cc|cc} \cline{1-5}
$\tau$& $\max_n \|x(t_n)-x_n\|$ & CR &  $\max_n \|v_1(t_n)-v_{1,n}\|$ &CR\\ 
\cline{1-5}		
$2^{-6}$&	1.23E-01&		&9.83E-02	&\\
$2^{-7}$&	6.00E-02&	1.04 &	4.94E-02&	0.99\\ 
$2^{-8}$&	2.92E-02&	1.04 &	2.44E-02&	1.02 \\
$2^{-9}$&	1.40E-02&	1.06 &	1.18E-02&	1.05 \\
				\hline	
			\end{tabular}}
			\label{table1:1}
		\end{table}

\begin{table}[h]
\setlength{\abovecaptionskip}{0pt}
\centering
\caption{Convergence rates of computing index-2 saddle point.}
	\resizebox{\textwidth}{!}{
\begin{tabular}{c|cc|cc|cc} \cline{1-7}
$\tau$& $\max_n \|x(t_n)-x_n\|$ & CR &  $\max_n \|v_1(t_n)-v_{1,n}\|$ &CR& $\max_n \|v_2(t_n)-v_{2,n}\|$ &CR\\ \cline{1-7}		
$2^{-6}$&	2.27E-01&		&1.43E-01	&	    &    1.43E-01&\\	
$2^{-7}$&	1.09E-01&	1.06 &	7.08E-02&	1.01 &	7.08E-02&	1.01\\ 
$2^{-8}$&	5.28E-02&	1.05 &	3.53E-02&	1.00 &	3.53E-02&	1.00 \\
$2^{-9}$&	2.53E-02&	1.06 &	1.72E-02&	1.03 &	1.72E-02&	1.03 \\
				\hline
			\end{tabular}}
			\label{table1:2}
		\end{table}

\textbf{Example 2: Comparison between two schemes}
We compare the behavior between the scheme (\ref{FDsadk}) and the scheme in \cite{Z3} by selecting the same initial values and parameters as in the previous example and computing $x_n$ and $X_n$ in Figure \ref{fig1}, which shows that both methods converge to the target saddle points along the same trajectory. 

To compare the dynamical behavior of these two methods in a pointwise-in-time manner, we plot $\|x_n-X_n\|$ and $\|v_{1,n}-V_{1,n}\|$ in the computation of the index-1 saddle point under different time-step size $\tau$ in Figure \ref{fig2}, which shows that the differences between the numerical solutions of these two methods are quite small at each time step, and such differences shrink as $\tau$ decreases. In particular, it seems from Figure \ref{fig2} that if $\tau$ becomes $\tau/2$, then the magnitudes of $\|x_n-X_n\|$ and $\|v_{1,n}-V_{1,n}\|$ also reduce by a half, which is consistent with the discussions in Remark \ref{rem1}.

\begin{figure}[h!]
	\setlength{\abovecaptionskip}{0pt}
	\centering	\includegraphics[width=2in,height=1.5in]{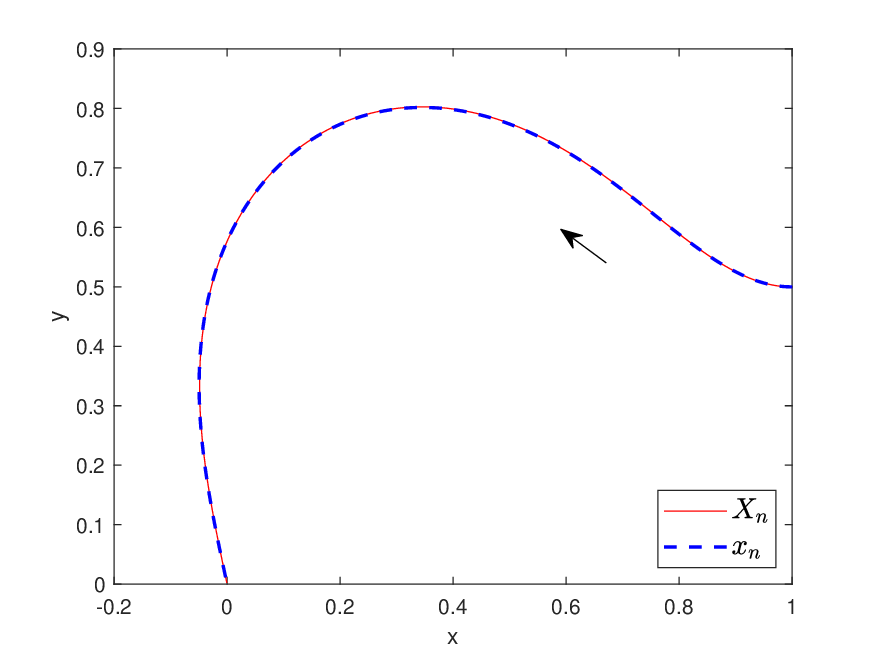}~~~~~~
	\includegraphics[width=2in,height=1.5in]{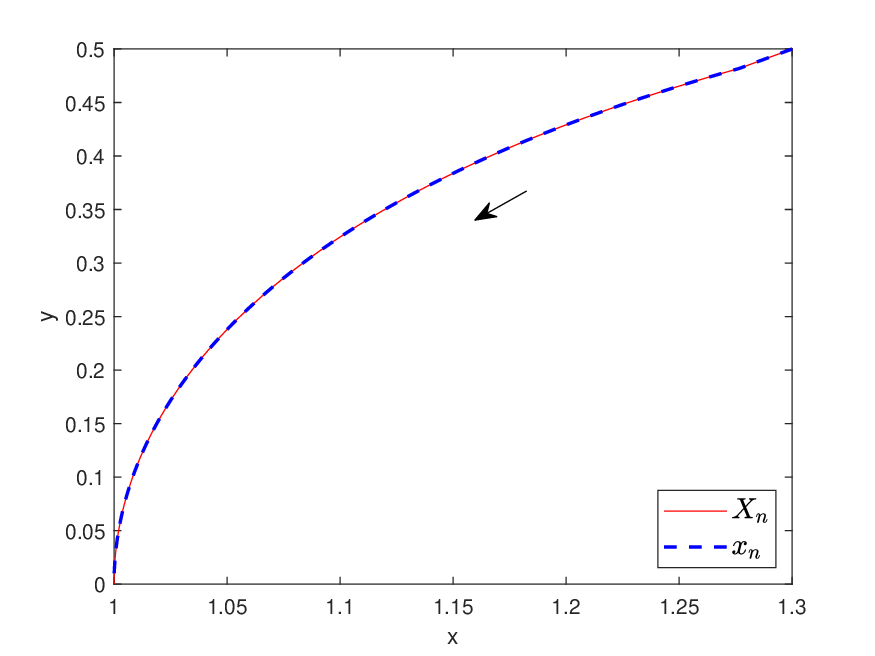}
	\caption{Convergence of numerical solutions $x_n$ and $X_n$ to (left) the index-1 saddle point and (right) the index-2 saddle point under $T=7$ and $\tau=1/100$.}
	\label{fig1}
\end{figure}

\begin{figure}[h!]
	\setlength{\abovecaptionskip}{0pt}
	\centering	\includegraphics[width=2in,height=1.5in]{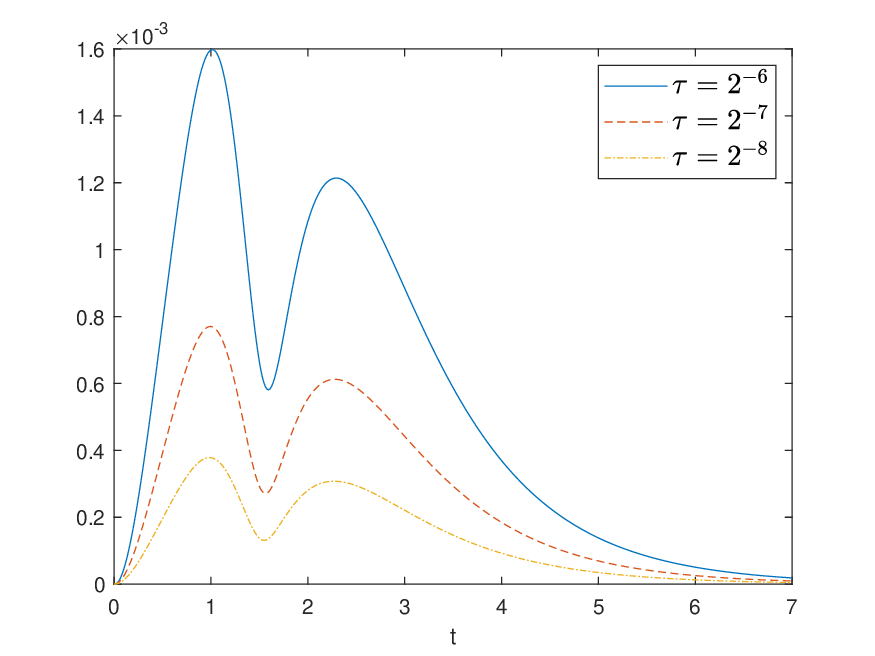}~~~~~~
	\includegraphics[width=2in,height=1.5in]{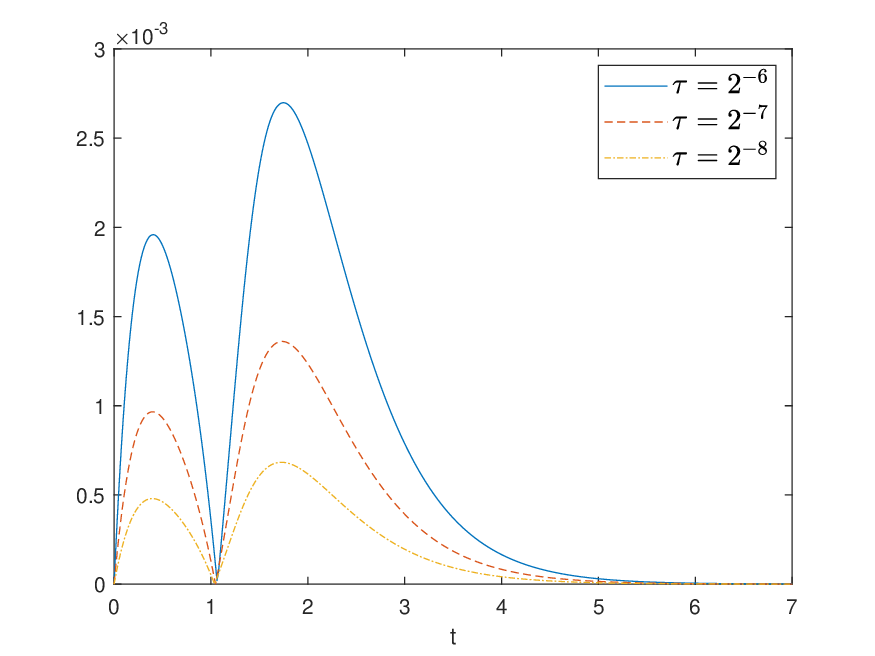}
	\caption{Plots of (left) $\|x_n-X_n\|$ and (right) $\|v_{1,n}-V_{1,n}\|$ under $T=7$ and different $\tau$ when computing the index-1 saddle point.}
	\label{fig2}
\end{figure}

\section{Extension to constrained HiSD}\label{sec5}
In this section we extend the developed methods and results for the constrained HiSD (\ref{csadk}) to substantiate the conclusions in \textbf{(ii)} in Section \ref{sec1.2}.

\subsection{Numerical discretization}
From the derivation of (\ref{csadk}) in \cite{CHiSD2021}, the nonlinear terms 
 \begin{equation}\label{omit1}
 -xx^\top F(x)\text{ and }  \mathcal L_i-xx^\top J(x)v_i+ xv_i^\top F(x) 
 \end{equation}
  in the equations of $x$ and $v_i$ are proposed to ensure the following manifold properties: if the following relations
\begin{equation}\label{prop}
 x\in S^{d-1},~~v_i^\top x=0,~~v_{i}^\top v_{j}=\delta_{ij},~~1\leq i,j\leq k
 \end{equation} 
hold at $t=0$, then they hold for any $t\geq 0$.
In practical computaitons, the following efficient numerical scheme of (\ref{csadk}) was proposed in \cite{CHiSD2021} for $1\leq n\leq N$
\begin{equation}\label{cFDsadk}
\left\{
\begin{array}{l}
\ds \tilde x_{n} =x_{n-1}+\tau\mathcal S^{n-1},\\[0.05in]
\ds x_n=\frac{\tilde x_n}{\|\tilde x_n\|},\\[0.15in]
\ds \tilde v_{i,n}=v_{i,n-1}+\tau \mathcal R_i^{n-1},\\[0.05in]
\hat v_{i,n}=\tilde v_{i,n}-\tilde v_{i,n}^\top x_n x_n,\\[0.05in]
\ds  v_{i,n}=\text{GramSchmidt}(v_{1,n},\cdots,v_{i-1,n};\hat v_{i,n}),~~1\leq i\leq k,
\end{array}
\right.
\end{equation}
equipped with the initial state $x_0\in S^{d-1}$ and orthonormal initial directional vectors $\{v_{i,0}\}_{i=1}^k$  such that $v_{i,0}^\top x_0=0$ for $1\leq i\leq k$.
Here the second equation of (\ref{cFDsadk}) represents the retraction in order to ensure that $x_n\in S^{d-1}$. The last two equations, which stand for the vector transport and the Gram-Schmidt orthonormalization procedure, respectively, aim to ensure the discrete analogue of (\ref{prop}), that is,
\begin{equation}\label{propn}
v_{i,n}^\top x_n=0,~~v_{i,n}^\top v_{j,n}=\delta_{ij},~~1\leq i,j\leq k,~~0\leq n\leq N.
\end{equation}
Furthermore, we use the explicit expression of the Gram-Schmidt orthonormalization as (\ref{vgs}) with $\tilde v_{i,n}$ and $Y_{i,n}$ in (\ref{vgs}) replaced by $\hat v_{i,n}$ and $Z_{i,n}$, respectively, for distinguishment.

\subsection{Recovery of constrained HiSD}
The main result of this section is to recover the schemes of $x$ and $\{v_i\}_{i=1}^k$ in the following theorem.
\begin{theorem}\label{thmc}
For $\tau$ small enough, the following relations could be derived from the scheme (\ref{cFDsadk})
\begin{equation}\label{thmceq}
\left\{
\begin{array}{l}
\ds  \frac{x_{n} -x_{n-1}}{\tau}=\mathcal S^{n-1}- x_{n-1}x_{n-1}^\top F(x_{n-1})+O(\tau),\\[0.125in]
\ds  \frac{v_{i,n}-v_{i,n-1}}{\tau}=\mathcal R^{n-1}_i+\mathcal L_i^{n-1}- x_{n-1}x_{n-1}^\top J(x_{n-1}) v_{i,n-1}\\[0.125in]
\ds\qquad\qquad+ x_{n-1} v_{i,n-1}^\top F(x_{n-1})+O(\tau),~~1\leq i\leq k.
\end{array}
\right.
\end{equation}
\end{theorem}
\begin{remark}
Similar to Remark \ref{rem31}, the recovered schemes in (\ref{thmceq}) have exactly the same forms as the continuous problem (\ref{csadk}) such that the statements in \textbf{(ii)} could be justified.
\end{remark}
\begin{proof} We prove this theorem in the following three steps.
\paragraph{Step 1: Derivation of the first equation in (\ref{thmceq})}

From the first equation of (\ref{cFDsadk}) we apply $x_{n-1}^\top v_{j,n-1}=0$ to obtain 
\begin{equation}\label{cxnorm}
\begin{array}{l}
\ds \|\tilde x_{n}\|^2=1+2\tau x_{n-1}^\top F(x_{n-1})+O(\tau^2).
\end{array}
\end{equation}
Then we use this and the second equation of (\ref{cFDsadk}) to obtain 
$$\begin{array}{l}
\ds x_n=\tilde x_n+\frac{1-\|\tilde x_n\|^2}{\|\tilde x_n\|(1+\|\tilde x_n\|)}\tilde x_n\\[0.15in]
\ds\quad\,\,=\tilde x_n+\frac{-2\tau x_{n-1}^\top F(x_{n-1})+O(\tau^2)}{\|\tilde x_n\|(1+\|\tilde x_n\|)}\tilde x_n\\[0.15in]
\ds\quad\,\,=\tilde x_n-\tau x_{n-1}^\top F(x_{n-1})\tilde x_n\\[0.05in]
\ds\qquad-2\tau x_{n-1}^\top F(x_{n-1})\bigg(\frac{1}{\|\tilde x_n\|(1+\|\tilde x_n\|)}-\frac{1}{2}\bigg)\tilde x_n+O(\tau^2).
\end{array} $$
Similar to the estimate (\ref{YY}), the third right-hand side term is an $O(\tau^2)$ term, while, based on the first equation of (\ref{cFDsadk}), the second right-hand side term could be rewritten as
$$-\tau x_{n-1}^\top F(x_{n-1})\tilde x_n= -\tau x_{n-1}^\top F(x_{n-1})x_{n-1}+O(\tau^2).$$
We incorporate the above equations to get
$$x_n=\tilde x_n-\tau x_{n-1}^\top F(x_{n-1})x_{n-1}+O(\tau^2), $$
which proves the first equation of (\ref{thmceq}). 

\paragraph{Step 2: A preliminary estimate of $v_{i,n}-v_{i,n-1}$}

The derivation of the second equation of (\ref{thmceq}) is much more complicated as we need to combine the last three equations in (\ref{cFDsadk}) by an appropriate manner. We first invoke the third equation of (\ref{cFDsadk}) in the forth equation to obtain
$$\begin{array}{l}
\ds \hat v_{i,n}=v_{i,n-1}+\tau J(x_{n-1}) v_{i,n-1}-(v_{i,n-1}+\tau J(x_{n-1}) v_{i,n-1})^\top x_n x_n.
\end{array} $$
From the first equation of (\ref{thmceq}) we have
$$v_{i,n-1}^\top x_n=-\tau v_{i,n-1}^\top F(x_{n-1})+O(\tau^2). $$
Combining the above two equations and applying  the substitution $ x_n=x_{n-1}+O(\tau)$ (cf. the first equation of (\ref{thmceq})) lead to
\begin{equation}\label{cvhv} \begin{array}{l}
\ds \hat v_{i,n}=v_{i,n-1}+\tau J(x_{n-1}) v_{i,n-1}+\tau v_{i,n-1}^\top F(x_{n-1})x_n      \\[0.05in]
\ds\qquad\qquad-\tau v_{i,n-1}^\top J(x_{n-1})^\top  x_n x_n+O(\tau^2)\\[0.05in]
\ds\quad~\,=v_{i,n-1}+\tau J(x_{n-1}) v_{i,n-1}+\tau v_{i,n-1}^\top F(x_{n-1})x_{n-1}      \\[0.05in]
\ds\qquad\qquad-\tau v_{i,n-1}^\top J(x_{n-1})^\top  x_{n-1} x_{n-1}+O(\tau^2)\\[0.05in]
\ds\quad~\,=:v_{i,n-1}+\mathcal L_{i,n-1}+O(\tau^2),\\[0.05in]
\ds \hspace{-0.2in}\mathcal L_{i,n-1}:=\tau J(x_{n-1}) v_{i,n-1}+\tau v_{i,n-1}^\top F(x_{n-1})x_{n-1} \\[0.05in]
\ds\qquad\qquad-\tau v_{i,n-1}^\top J(x_{n-1})^\top  x_{n-1} x_{n-1}.
\end{array} \end{equation}
We invoke this equation in the last equation of (\ref{cFDsadk}) to obtain
\begin{equation}\label{vz}
\begin{array}{l}
\ds v_{i,n}=\frac{1}{Z_{i,n}}\bigg(v_{i,n-1}+\mathcal L_{i,n-1}+O(\tau^2)\\[0.15in]
\ds\qquad\qquad-\sum_{j=1}^{i-1} v_{j,n}^\top \big(v_{i,n-1}+\mathcal L_{i,n-1}+O(\tau^2)\big) v_{j,n}\bigg)\\
\ds\quad~\,=\frac{1}{Z_{i,n}}\bigg(v_{i,n-1}-\sum_{j=1}^{i-1} (v_{j,n}-v_{j,n-1})^\top v_{i,n-1} v_{j,n}+O(\tau)\bigg),
\end{array}
\end{equation}
which implies
\begin{equation}
\begin{array}{l}
\ds v_{i,n}-v_{i,n-1}\\
\ds\qquad=\frac{1}{Z_{i,n}}\bigg(v_{i,n-1}(1-Z_{i,n})-\sum_{j=1}^{i-1} (v_{j,n}-v_{j,n-1})^\top v_{i,n-1} v_{j,n}+O(\tau)\bigg).
\end{array}
\end{equation}
We also employ (\ref{cvhv}) to expand $Z_{i,n}$ as
\begin{equation}\label{ZZ}
\begin{array}{l}
\ds Z_{i,n}=\bigg(1+2\tau v_{i,n-1}^\top J(x_{n-1})v_{i,n-1}-\sum_{j=1}^{i-1} \big( v_{j,n}^\top v_{i,n-1} \big)^2+O(\tau^2)\bigg)^{1/2}\\
\ds\qquad=\bigg(1+2\tau v_{i,n-1}^\top J(x_{n-1})v_{i,n-1}\\
\ds\qquad\qquad-\sum_{j=1}^{i-1} \big( (v_{j,n}-v_{j,n-1})^\top v_{i,n-1} \big)^2+O(\tau^2)\bigg)^{1/2}\\
\ds\qquad=\bigg(1-\sum_{j=1}^{i-1} \big( (v_{j,n}-v_{j,n-1})^\top v_{i,n-1} \big)^2+O(\tau)\bigg)^{1/2}.
\end{array}
\end{equation}
Based on these equations we could follow the proof of Lemma \ref{lemvv} to prove that
\begin{equation}\label{cv-v}
\|v_{i,n}-v_{i,n-1}\|\leq Q\tau,~~1\leq i\leq k,~~1\leq n\leq N.
\end{equation}

\paragraph{Step 3: Derivation of the second equation in (\ref{thmceq})}
We invoke the estimate (\ref{cv-v}) back to the second equality of (\ref{ZZ}) to get
\begin{equation}\label{cZ1}
\ds Z_{i,n}=\big(1+2\tau v_{i,n-1}^\top J(x_{n-1})v_{i,n-1}+O(\tau^2)\big)^{1/2},
\end{equation}
which, together with (\ref{split}) and (\ref{YY}), implies
$$\frac{1}{Z_{i,n}}=1-\tau v_{i,n-1}^\top J(x_{n-1})v_{i,n-1}+O(\tau^2). $$
We replace $1/Z_{i,n}$ in the first equality of (\ref{vz}) by this equation to get
\begin{equation}\label{vz2}
\begin{array}{l}
\ds v_{i,n}=v_{i,n-1}+\mathcal L_{i,n-1}-\sum_{j=1}^{i-1} v_{j,n}^\top \big(v_{i,n-1}+\mathcal L_{i,n-1}\big) v_{j,n}\\[0.2in]
\ds\qquad\quad-\tau v_{i,n-1}^\top J(x_{n-1})v_{i,n-1}v_{i,n-1}\\[0.05in]
\ds\qquad\quad+\tau\sum_{j=1}^{i-1} v_{j,n}^\top v_{i,n-1} (v_{i,n-1}^\top J(x_{n-1})v_{i,n-1}) v_{j,n}+O(\tau^2).
\end{array}
\end{equation}
By (\ref{cv-v}), the last-but-one right-hand side term of (\ref{vz2}) could be estimated as
$$\begin{array}{l}
\ds \tau\sum_{j=1}^{i-1} v_{j,n}^\top v_{i,n-1} (v_{i,n-1}^\top J(x_{n-1})v_{i,n-1}) v_{j,n}\\
\ds\quad=\tau\sum_{j=1}^{i-1} (v_{j,n}-v_{j,n-1})^\top v_{i,n-1} (v_{i,n-1}^\top J(x_{n-1})v_{i,n-1}) v_{j,n}\\[0.175in]
\ds\quad=O(\tau^2),
\end{array} $$
and we reformulate the first summation on the right-hand side of (\ref{vz2}) as
\begin{equation}\label{LL} \begin{array}{l}
\ds -\sum_{j=1}^{i-1} v_{j,n}^\top \big(v_{i,n-1}+\mathcal L_{i,n-1}\big) v_{j,n}\\
\ds\qquad=-\sum_{j=1}^{i-1} (v_{j,n}-\hat v_{j,n})^\top v_{i,n-1} v_{j,n}-\sum_{j=1}^{i-1} \hat v_{j,n}^\top v_{i,n-1} v_{j,n}\\
\ds\qquad~~\,-\sum_{j=1}^{i-1} (v_{j,n}-v_{j,n-1})^\top \mathcal L_{i,n-1} v_{j,n}-\sum_{j=1}^{i-1} v_{j,n-1}^\top \mathcal L_{i,n-1} v_{j,n}\\[0.175in]
\qquad\ds=:\sum_{l=1}^4 B_l.
\end{array} \end{equation}
To bound $B_1$, we apply the last equation of (\ref{cFDsadk}) to get
$$\begin{array}{l}
\ds v_{j,n}-\hat v_{j,n}=\frac{1-Z_{j,n}}{Z_{j,n}}\hat v_{j,n}-\frac{1}{Z_{j,n}}\sum_{l=1}^{j-1}\hat v_{j,n}^\top v_{l,n} v_{l,n},
\end{array} $$
which implies
\begin{equation}\label{cvtv} \begin{array}{l}
\ds (v_{j,n}-\hat v_{j,n})^\top v_{i,n-1}=\frac{1-Z_{j,n}}{Z_{j,n}}\hat v_{j,n}^\top v_{i,n-1}-\frac{1}{Z_{j,n}}\sum_{l=1}^{j-1}\hat v_{j,n}^\top v_{l,n} v_{l,n}^\top v_{i,n-1}.
\end{array} \end{equation}
From (\ref{cZ1}) we find that $1-Z_{j,n}$ is an $O(\tau)$ term, and $\hat v_{j,n}^\top v_{i,n-1}$ could be expanded as
\begin{equation}\label{cs1}
\begin{array}{l}
\ds \hat v_{j,n}^\top v_{i,n-1}=(\tilde v_{j,n}-\tilde v_{j,n}^\top x_n x_n)^\top v_{i,n-1}\\[0.05in]
\ds\quad= (v_{j,n-1}+\tau J(x_{n-1}) v_{j,n-1})^\top v_{i,n-1}\\[0.05in]
\ds\qquad -\tilde v_{j,n}^\top x_n x_n^\top (v_{i,n-1}-v_{i,n})\\[0.05in]
\ds\quad= (\tau J(x_{n-1}) v_{j,n-1})^\top v_{i,n-1} \\[0.05in]
\ds\qquad-\tilde v_{j,n}^\top x_n x_n^\top (v_{i,n-1}-v_{i,n})=O(\tau).
\end{array}
\end{equation}
Thus the first right-hand side term of (\ref{cvtv}) is $O(\tau^2)$. The second right-hand side term of (\ref{cvtv}) could be reformulated as
\begin{equation}\label{hg} \begin{array}{l}
\ds -\frac{1}{Z_{j,n}}\sum_{l=1}^{j-1}\hat v_{j,n}^\top v_{l,n} v_{l,n}^\top v_{i,n-1}\\[0.15in]
\ds\qquad=-\frac{1}{Z_{j,n}}\sum_{l=1}^{j-1}\hat v_{j,n}^\top (v_{l,n}-v_{l,n-1}) v_{l,n}^\top (v_{i,n-1}-v_{i,n})\\[0.15in]
\ds\qquad\quad-\frac{1}{Z_{j,n}}\sum_{l=1}^{j-1}\hat v_{j,n}^\top v_{l,n-1} v_{l,n}^\top (v_{i,n-1}-v_{i,n}).
\end{array} \end{equation}
By (\ref{cv-v}) the first right-hand side term of this equation is an $O(\tau^2)$ term, while, by a similar derivation as (\ref{cs1}), the factor $\hat v_{j,n}^\top v_{l,n-1}$ in the second right-hand side term of this equation is an $O(\tau)$ term, which implies that the second right-hand side term of (\ref{hg}) and thus (\ref{cvtv}) are $O(\tau^2)$. Consequently, $B_1$ in (\ref{LL}) is $O(\tau^2)$. 

To estimate $B_2$, we find that 
$$\begin{array}{l}
\ds \tilde v_{j,n}^\top x_n=(v_{j,n-1}+\tau J(x_{n-1}) v_{j,n-1})^\top x_n\\[0.05in]
\ds\qquad\qquad=(v_{j,n-1}-v_{j,n})^\top x_n+\tau ( J(x_{n-1}) v_{j,n-1})^\top x_n=O(\tau),
\end{array}  $$
which, together with the third inequality of (\ref{cs1}), implies
\begin{equation*}
\ds \hat v_{j,n}^\top v_{i,n-1}= (\tau J(x_{n-1}) v_{j,n-1})^\top v_{i,n-1} +O(\tau^2).
\end{equation*}
Thus we could rewrite $B_2$ as
$$\begin{array}{l}
\ds B_2=-\sum_{j=1}^{i-1} (\tau J(x_{n-1}) v_{j,n-1})^\top v_{i,n-1}  v_{j,n}+O(\tau^2)\\[0.15in]
\ds\quad~=-\sum_{j=1}^{i-1} (\tau J(x_{n-1}) v_{j,n-1})^\top v_{i,n-1}  v_{j,n-1}\\[0.15in]
\ds\qquad-\sum_{j=1}^{i-1} (\tau J(x_{n-1}) v_{j,n-1})^\top v_{i,n-1} (v_{j,n}- v_{j,n-1})+O(\tau^2)\\[0.15in]
\ds\quad~=-\sum_{j=1}^{i-1} (\tau J(x_{n-1}) v_{j,n-1})^\top v_{i,n-1}  v_{j,n-1}+O(\tau^2).
\end{array}  $$
$B_3$ is clearly an $O(\tau^2)$ term, and we expand $B_4$ as
 $$\begin{array}{l}
 \ds B_4=-\sum_{j=1}^{i-1} v_{j,n-1}^\top \big(\tau J(x_{n-1}) v_{i,n-1}+\tau v_{i,n-1}^\top F(x_{n-1})x_{n-1} \\[0.2in]
\ds\qquad\qquad-\tau v_{i,n-1}^\top J(x_{n-1})^\top  x_{n-1} x_{n-1}\big) v_{j,n}\\[0.05in]
\ds\quad~=-\sum_{j=1}^{i-1} v_{j,n-1}^\top \tau J(x_{n-1}) v_{i,n-1}v_{j,n}\\[0.2in]
\ds\quad~=-\sum_{j=1}^{i-1} v_{j,n-1}^\top \tau J(x_{n-1}) v_{i,n-1}v_{j,n-1}\\[0.2in]
\ds\qquad\qquad-\sum_{j=1}^{i-1} v_{j,n-1}^\top \tau J(x_{n-1}) v_{i,n-1}(v_{j,n}-v_{j,n-1})\\[0.2in]
\ds\quad~=-\sum_{j=1}^{i-1} v_{j,n-1}^\top \tau J(x_{n-1}) v_{i,n-1}v_{j,n-1}+O(\tau^2).
\end{array}   $$
Invoking the estimates of $B_1-B_4$ in (\ref{LL}) leads to
$$-\sum_{j=1}^{i-1} v_{j,n}^\top \big(v_{i,n-1}+\mathcal L_{i,n-1}\big) v_{j,n}=-2\tau\sum_{j=1}^{i-1} v_{j,n-1}^\top  J(x_{n-1}) v_{i,n-1}v_{j,n-1}+O(\tau^2), $$
and we incorporate this equation with (\ref{vz2}) to obtain the second equation of (\ref{thmceq}), which completes the proof.
\end{proof}

Based on this theorem, the error equations could be generated by subtracting the reference equations of (\ref{csadk}) from (\ref{thmceq}), which, together with the conventional numerical analysis method for systems of differential equations, lead to the following error estimate of the numerical scheme (\ref{cFDsadk}).
\begin{theorem}
Under the Assumption A, the following estimate holds for the numerical scheme (\ref{cFDsadk}) for $\tau$ sufficiently small 
$$\|x_{n}-x(t_n)\|+\sum_{i=1}^k \|v_{i,n}-v_i(t_n)\|\leq Q\tau,~~1\leq n\leq N. $$
Here $Q$ is independent from $\tau$, $n$ and $N$.
\end{theorem}

\section{Concluding remarks}
In this paper we analyze an efficient discrete HiSD scheme, which drop the Lagrangian multiplier terms in HiSD and instead perform an additional Gram-Schmidt orthonormalization to ensure the Stiefel manifold constraint. We recover the same form as HiSD from this scheme, which not only generates error estimates naturally, but indicates that the mechanism of Stiefel manifold preservation in HiSD is nearly a Gram-Schmidt process. The developed methods are further extended to analyze the more complicated constrained HiSD on high-dimensional sphere, which reveal that the mechanisms of the constrained HiSD on preserving several manifold properties are close to simple operations such as the vector normalization, the vector transport and the Gram-Schmidt orthonormalization. These results reveal mechanisms of the HiSD and constrained HiSD in preserving several manifold properties via numerical analysis.
   
There are several other potential extensions of the current work that deserve further exploration. For instance, one could apply the projection method proposed in \cite[Example 4.6]{Haibook} instead of the Gram-Schmidt process in (\ref{FDsadk}) and (\ref{cFDsadk}) to retract the directional vectors back to the Stiefel manifold, which preserves the manifold property via the minimal adjustment. Specifically, let $\tilde V_n:=[\tilde v_{n,1},\cdots,\tilde v_{n,k}]\in\mathbb R^{n\times k}$, then the projection $V_n:=[v_{n,1},\cdots,v_{n,k}]$ could be determined by minimizing the Frobenius norm of the difference $\tilde V_n-V_n$ within the Stiefel manifold, i.e.
\begin{equation}\label{min}
\text{min }\|\tilde V_n-V_n\|_F~~\text{ subject to }V_n^\top V_n=I. \end{equation}
In practice, one could compute the singular value decomposition
$\tilde V_n=U^\top \Sigma W$ and then the solution to (\ref{min}) is $V_n=U^\top W$. Thus it is natural to consider to what extend the application of the projection method deviates the dynamics of the HiSD, or the possibility of designing a new form of HiSD whose mechanism of preserving the Stiefel manifold is nearly the projection method. As the solution of (\ref{min}) does not have a clear form as the Gram-Schmidt process, more investigations are required to analyze these questions.

Furthermore, the ideas and techniques could be employed and improved to perform numerical analysis for HiSD constrained by $m$ equalities \cite[Equation 24]{CHiSD2021}
\begin{equation}\label{ccsd}
\left\{
\begin{array}{l}
\ds \frac{dx}{dt} =\mathcal S(t),\\[0.075in]
\ds \frac{dv_i}{dt}=\bigg(I-v_iv_i^\top-2\sum_{j=1}^{i-1}v_jv_j^\top\bigg)\mathcal H(x)[v_i] \\
\ds\qquad\qquad-A(x)\big(A(x)^\top A(x)\big)^{-1}\bigg(\nabla^2c(x)\frac{dx}{dt}\bigg)^\top v_i,~~1\leq i\leq k.
\end{array}
\right.
\end{equation}
Here $c(x)=(c_1(x),\cdots,c_m(x))=0$ represents the $m$ equality constraints and $A(x)=(\nabla c_1(x),\cdots,\nabla c_m(x))$. The constrained HiSD (\ref{csadk}) is a special case of (\ref{ccsd}) with one equality constraint
$$c_1(x)=\|x\|^2-1=0. $$
In this generalized constrained HiSD (\ref{ccsd}), $\mathcal H(x)$ refers to the Riemannian Hessian \cite{CHiSD2021}, which is difficult to compute and approximate in practice. Furthermore, compared with (\ref{csadk}), additional complicated terms appear on the right-hand side of (\ref{ccsd}). These bring additional difficulties for the numerical analysis that we will investigate in the near future.

\section*{Acknowledgements}
This work was partially supported by the National
Natural Science Foundation of China (No.
12288101, 12225102, T2321001,  12301555),
the Taishan Scholars Program of Shandong
Province (No. tsqn202306083), the National Key R\&D Program of China (No. 2023YFA1008903).

          \end{document}